\documentclass[12pt]{article}
\usepackage{times}
\usepackage{booktabs}
\usepackage{pifont}
\usepackage{caption}
\usepackage{mathrsfs}
\usepackage[fleqn]{amsmath}
\usepackage{amsfonts,amsthm,amssymb,mathrsfs,bbding}
\usepackage[varg]{txfonts}
\usepackage{graphics,multicol}
\usepackage{graphicx}
\usepackage{color}
\usepackage{enumerate}
\usepackage{caption}
\usepackage{tabularx}
\usepackage[
colorlinks,
anchorcolor=blue,
linkcolor=blue,
citecolor=blue]{hyperref}
\usepackage{cite}
\usepackage{latexsym,bm}
\usepackage{mathtools}
\evensidemargin=\oddsidemargin\topmargin=-1.5cm

\newtheorem{thm}{Theorem}[section]

\newtheorem{lem}[thm]{Lemma}

\newtheorem{remark}{Remark}

\theoremstyle{definition}

\addtocounter{section}{0}
\begin{document}
\title{Automorphism groups of a class of cubic Cayley graphs on symmetric groups\footnote{This work is supported by the National Natural Science Foundation of China (Grant Nos. 11671344, 11261059 and 11531011).}}
\author{{\small Xueyi Huang, \ \ Qiongxiang Huang\footnote{
Corresponding author. E-mail: huangqx@xju.edu.cn, huangqxmath@163.com.}, \ \ Lu Lu}\\[2mm]\scriptsize
College of Mathematics and Systems Science,
\scriptsize Xinjiang University, Urumqi, Xinjiang 830046, P. R. China}
\date{}
\maketitle
{\flushleft\large\bf Abstract}
Let $S_n$ denote the symmetric group  of degree $n$ with $n\geq 3$. Set $S=\{c_n=(1\ 2\ldots \ n),c_n^{-1},(1\ 2)\}$.  Let $\Gamma_n=\mathrm{Cay}(S_n,S)$ be the Cayley graph on $S_n$ with respect to $S$. In this paper,  we show that $\Gamma_n$ ($n\geq 13$) is  a normal Cayley graph, and that the full automorphism group of $\Gamma_n$ is equal to $\mathrm{Aut}(\Gamma_n)=R(S_n)\rtimes \langle\mathrm{Inn}(\phi)\rangle\cong S_n\rtimes \mathbb{Z}_2$, where $R(S_n)$ is the right regular representation of $S_n$, $\phi=(1\ 2)(3\ n)(4\ n-1)(5\ n-2)\cdots$ $(\in S_n)$, and $\mathrm{Inn}(\phi)$ is the inner isomorphism of $S_n$ induced by $\phi$.
\begin{flushleft}
\end{flushleft}

\section{Introduction}\label{s-1}
Let $G$ be a finite group, and $S$ a subset of $G$ with $e\not\in S$ ($e$ is the identity element of $G$) and  $S=S^{-1}$. The \emph{Cayley graph} on $G$ with respect to $S$, denoted by $\mathrm{Cay}(G,S)$,  is defined to be the undirected graph with vertex set $G$, and with an edge connecting $g,h\in G$ if $hg^{-1}\in S$. Denote by $\mathrm{Aut}(\mathrm{Cay}(G,S))$ and $\mathrm{Aut}(G)$ the automorphism groups of $\mathrm{Cay}(G,S)$ and $G$, respectively. The \emph{right regular representation} of the group $G$ is defined as $R(G)=\{r_g:x\mapsto xg~(\forall x\in G)\mid g\in G\}$. Clearly, $R(G)$ is a subgroup of $\mathrm{Aut}(\mathrm{Cay}(G,S))$ and so every Cayley graph is vertex-transitive. Furthermore, the group $\mathrm{Aut}(G, S)=\{\sigma\in \mathrm{Aut}( G)\mid S^\sigma=S\}$ is  a subgroup of $\mathrm{Aut}(\mathrm{Cay}(G,S))_e$, the
stabilizer of the identity vertex $e$ in $\mathrm{Aut}(\mathrm{Cay}(G,S))$, and so is also a subgroup of  $\mathrm{Aut}(\mathrm{Cay}(G,S))$. The Cayley graph $\mathrm{Cay}(G,S)$ is said to be
\emph{normal} if $R(G)$ is a normal subgroup of $\mathrm{Aut}(\mathrm{Cay}(G,S))$. Godsil in \cite{Godsil} proved that   $N_{\mathrm{Aut}(\mathrm{Cay}(G,S))}(R(G)) = R(G)\rtimes \mathrm{Aut}(G, S)$, which implies that $\mathrm{Cay}(G,S)$ is normal if and only if $\mathrm{Aut}(\mathrm{Cay}(G,S))= R(G)\rtimes \mathrm{Aut}(G, S)$.

To determine the full automorphism groups of Cayley graphs is a basic problem in algebraic graph theory. As normal Cayley graphs are just those which have the smallest possible full automorphism groups, to determine  the normality of  Cayley graphs  is an important problem in the literature \cite{Xu}. The whole information about the normality of Cayley graphs on the cyclic groups of prime order, the groups of order twice a prime, a prime-square and a product of two distinct primes were obtained by Alspach \cite{Alspach}, Du et al. \cite{Du}, Dobson et al. \cite{Dobson} and Lu et al. \cite{Lu}, respectively.  For more results regarding automorphism groups and normality of Cayley graphs, we refer the reader to  \cite{Feng2,Xu} and  references therein.

Let $S_n$ and $A_n$ denote the symmetric group  and alternating group of degree $n$, respectively.  In the past few years the problem of determining the full automorphism groups of Cayley graphs on $S_n$ and $A_n$ has received considerable attention (see, for example, \cite{Deng,Deng1,Feng1,Huang,Ganesan,Ganesan1,Korchmaros,Zhang1,Zhou}).  This is mainly due to the fact that the Cayley graphs, especially those on $S_n$ and $A_n$, are widely used as models for interconnection networks \cite{Jwo,Lakshmivarahan}. It is well known that  $S=\{c_n=(1\ 2\ldots \ n),c_n^{-1},(1\ 2)\}$ can generate $S_n$. Thus $\Gamma_n=\mathrm{Cay}(S_n,S)$ is a connected graph. The directed graph $\mathrm{Cay}(S_n,S'=\{c_n, (1\ 2)\})$ for even $n\geq 4$ has been used to provide an infinite family of non-hamiltonian directed Cayley graphs (see \cite{Godsil1}, Corollary 3.8.2). It motivates us to consider the problem of determining the full automorphism group of $\Gamma_n=\mathrm{Cay}(S_n,S)$.

In this paper, it is shown that $\Gamma_n$ ($n\geq 13$) is a normal Cayley graph, and that the full automorphism group of $\Gamma_n$ is equal to $\mathrm{Aut}(\Gamma_n)=R(S_n)\rtimes\mathrm{Aut}(S_n,S)=R(S_n)\rtimes \langle\mathrm{Inn}(\phi)\rangle\cong S_n\rtimes \mathbb{Z}_2$, where $R(S_n)$ is the right regular representation of $S_n$, $\phi=(1\ 2)(3\ n)(4\ n-1)(5\ n-2)\cdots$ $(\in S_n)$, and $\mathrm{Inn}(\phi)$ is the inner isomorphism of $S_n$ induced by $\phi$. Besides, we also provide the full automorphism group of $\Gamma_n$ for $3\leq n\leq 8$ with the help of the package ``grape'' of GAP4 \cite{GAP4}.

\section{Main Results}\label{s-2}
The main goal of this section is to determine the full automorphism group of $\Gamma_n$. First of all, we need the following crucial criterion for a Cayley graph to be normal.
\begin{lem}[\cite{Xu}]\label{lem-1}
Let $\mathrm{Cay}(G,S)$ be the Cayley graph on $G$ with respect to $S$.  Then $\mathrm{Cay}(G,S)$ is normal
if and only if $\mathrm{Aut}(\mathrm{Cay}(G,S))_e=\mathrm{Aut}(G,S)$.
\end{lem}
For connected Cayley graphs, the above criterion could be simplified as follows, which is well-known and  easily verified by oneself.

\begin{lem}\label{lem-2}
Let $G$ be a finite group, and let $S$ ($e\not\in S$) be a symmetric generating set of $G$. Then $\mathrm{Aut}(\mathrm{Cay}(G,S))_e=\mathrm{Aut}(G,S)$ if and only if $(st)^\sigma=s^\sigma t^\sigma$ holds for any $\sigma\in \mathrm{Aut}(\mathrm{Cay}(G,S))_e$ and  $s,t\in S$.
\end{lem}

The following lemma gives the automorphism group of $S_n$, which is useful for us to determine the full automorphism group of $\Gamma_n$.
\begin{lem}[\cite{Suzuki}, Chapter 3, Theorems 2.17--2.20]\label{lem-3}
If $n \geq 3$ and $n\neq 6$, then $\mathrm{Aut}(S_n)=\mathrm{Inn}(S_n)\cong S_n$.  If $n = 6$, then $|\mathrm{Aut}(S_n):\mathrm{Inn}(S_n)|=2$, and each element in $\mathrm{Aut}(S_n)\setminus\mathrm{Inn}(S_n)$ maps a transposition to a product of three disjoint transpositions.
\end{lem}

The following two lemmas provide a main tool  for us to prove the normality of $\Gamma_n$.
\begin{lem}\label{lem-4}
Let $S=\{c_n=(1\ 2\ldots\ n),c_n^{-1},(1\ 2)\}$ and $\Gamma_n=\mathrm{Cay}(S_n,S)$ ($n\geq 13$). Then there is an unique $12$-cycle in $\Gamma_n$ passing through $e$, $c_n$ and $c_n^{-1}$ which is shown in (\ref{12-cycle-6}).
\end{lem}
\begin{proof}
Assume that $C=(e,c_n,s_1c_n,s_2s_1c_n,\ldots,s_{10}\cdots s_2s_1c_n=c_n^{-1},c_{n}s_{10}$ $\cdots s_2s_1c_n=e)$ is an arbitrary $12$-cycle in $\Gamma_n$ passing through $e$, $c_n$ and $c_n^{-1}$, where $s_i\in S$ for $1\leq i\leq 10$. Since $C$ is  a cycle, we have $s_1\neq c_n^{-1}$, $s_i\neq s_{i-1}^{-1}$ for $2\leq i\leq 10$ and $s_{10}\neq c_n^{-1}$. Thus there exists a positive integer $k$ such that $C$  is determined by the  equation
\begin{equation}\label{12-cycle-1}
e=c_{n}s_{10}\cdots s_2s_1c_n=c_n^{i_k}(1\ 2)c_n^{i_{k-1}}(1\ 2)\cdots c_n^{i_3}(1\ 2)c_n^{i_2}(1\ 2)c_n^{i_1},
\end{equation}
where $i_1\ge 1$, $i_k\ge 1$, $|i_l|\ge1$ for $l=2,3,\ldots,k-1$ and $i_1+|i_2|+\cdots+|i_{k-1}|+i_k+k-1=12$.
Clearly, we have $k\le 6$. If $k=1$, then (\ref{12-cycle-1}) is equivalent to $e=c_n^{12}$, which is impossible because $n\geq 13$. Therefore, we have $2\leq k\leq 6$. Let $u_0=(1\ 2)$ and  $u_l=c_n^{-i_1-i_2-\cdots-i_{1-1}-i_l}(1\ 2)c_n^{i_1+i_2+\cdots+i_{l-1}+i_l}$ for $l=1,2,\ldots,k-1$. Then  $u_l=c_{n}^{-i_l}u_{l-1}c_{n}^{i_l}$ for $l=1,\ldots,k-1$, and so (\ref{12-cycle-1}) becomes
\begin{equation}\label{12-cycle-2}
e=c_n^{i_1+i_2+\cdots+i_{k-1}+i_k}u_{k-1}\cdots u_3u_2u_1,
\end{equation}
or equivalently,
\begin{equation}\label{12-cycle-3}
c_n^{-i_1-i_2-\cdots-i_{k-1}-i_k}=u_{k-1}\cdots u_3u_2u_1.
\end{equation}

Note that $i_1+i_2+\cdots+i_{k-1}+i_k\in[-\sum_{l=1}^k|i_l|,\sum_{l=1}^k|i_l|]\subseteq[-12,12]\subset[-n,n]$ because $\sum_{l=1}^k|i_l|\leq 12$ and $n\geq 13$. Therefore, if $i_1+i_2+\cdots+i_{k-1}+i_k\not=0$, then $i_1+i_2+\cdots+i_{k-1}+i_k\not\equiv0$ $(\mathrm{mod}~n)$, and so $|\mathrm{supp}(c_n^{-i_1-i_2-\cdots-i_{k-1}-i_k})|=|\mathrm{supp}(c_n^{i_1+i_2+\cdots+i_{k-1}+i_k})|=n\geq 13$. However, $|\mathrm{supp}(u_{k-1}\cdots u_3u_2u_1)|\leq 2(k-1)\leq 10$, which is a contradiction according to (\ref{12-cycle-3}). Thus $i_1+i_2+\cdots+i_{k-1}+i_k=0$. From (\ref{12-cycle-2}) we have
\begin{equation}\label{12-cycle-4}
u_{k-1}\cdots u_3u_2u_1=e.
\end{equation}
Clearly, $k=3$ or $5$ since $2\leq   k\leq 6$ and $e$ is an even permutation.

If $k=3$, then $i_1+|i_2|+i_3=10$ and $i_1+i_2+i_3=0$, implying that $i_2=-5$, and $\{i_1,i_3\}=\{2,3\}$ or $\{1,4\}$ because $i_1,i_3\geq 1$.  In the former case, if $i_1=2$ then $u_1=c_{n}^{-2}(1\ 2)c_{n}^{2}=(3\ 4)$ and $u_2=c_{n}^{-i_2}u_1c_{n}^{i_2} =c_{n}^{5}(3\ 4)c_{n}^{-5}=(n-2\ n-1)$, which gives that $u_2u_1=(n-2\ n-1)(3\ 4)\not=e$, contrary to (\ref{12-cycle-4}); if $i_1=3$, similarly, we have  $u_2u_1=(n-1\ n)(4\ 5)\not=e$, also contrary to (\ref{12-cycle-4}). In the later case, one can also deduce a contradiction in the same way.

If  $k=5$, then $i_1+|i_2|+|i_3|+|i_{4}|+i_5=8$ and $i_1+i_2+i_3+i_{4}+i_5=0$. Since $i_1,i_5\geq 1$, it is easy to see that there are seven types of solutions satisfying these conditions:
\begin{equation*}
\left\{\begin{array}{rl}
\mbox{I-type:}&\{i_1,i_5\}=\{1,1\}, ~\{i_2,i_3,i_4\}=\{-4,1,1\};\\
\mbox{II-type:}&\{i_1,i_5\}=\{1,1\}, ~\{i_2,i_3,i_4\}=\{-3,-1,2\};\\
\mbox{III-type:}&\{i_1,i_5\}=\{1,1\}, ~\{i_2,i_3,i_4\}=\{-2,-2,2\};\\
\mbox{IV-type:}&\{i_1,i_5\}=\{2,1\},  ~\{i_2,i_3,i_4\}=\{-3,-1,1\};\\
\mbox{V-type:}&\{i_1,i_5\}=\{2,1\},  ~\{i_2,i_3,i_4\}=\{-2,-2,1\};\\
\mbox{VI-type:}&\{i_1,i_5\}=\{2,2\},  ~\{i_2,i_3,i_4\}=\{-2,-1,-1\};\\
\mbox{VII-type:}&\{i_1,i_5\}=\{3,1\},  ~\{i_2,i_3,i_4\}=\{-2,-1,-1\}.\\
\end{array}\right.
\end{equation*}

For I-type and II-type, since $i_1=i_5=1$, we have $u_1=c_{n}^{-1}(1\ 2)c_n=(2\ 3)$ and $u_4=c_n^{-i_1-i_2-i_3-i_4}(1\ 2)$ $c_n^{i_1+i_2+i_3+i_4}=c_n^{i_5}(1\ 2)c_n^{-i_5}=c_n(1\ 2)c_n^{-1}=(n\ 1)$, which gives that $\{u_2, u_3\}=\{(2\ 3),(n\ 1)\}$ due to $u_4u_3u_2u_1=e$. If $u_2=(2\ 3)$,  from $u_2=c_n^{-i_2}u_1c_n^{i_2}$ we deduce that $i_2=0$, a  contradiction. If  $u_2=(n\ 1)$, similarly, one can deduce that $i_2=-2$, which is impossible because $\{i_1,i_2,i_3\}=\{-4,1,1\}$  or $\{-3,-1,2\}$.

For III-type, we also have $u_1=(2\ 3)$, $u_4=(n\ 1)$ and $\{u_2, u_3\}=\{(2\ 3),(n\ 1)\}$. If $u_2=(2\ 3)$,  from $u_2=c_n^{-i_2}u_1c_n^{i_2}$ we get $i_2=0$, a  contradiction. If  $u_2=(n\ 1)$, then $u_3=(2,3)$. This implies that $i_2=-2$, $i_3=2$ and $i_4=-2$, and so (\ref{12-cycle-1}) becomes
\begin{equation}\label{12-cycle-5}
c_n(1\ 2)c_n^{-2}(1\ 2)c_n^2(1\ 2)c_n^{-2}(1\ 2)c_n=e,
\end{equation}
which holds naturally because $(1\ 2)c_n^{-2}(1\ 2)c_n^2=(1\ 2)(3\ 4)$ is of order $2$. This leads to a possible $12$-cycle in $\Gamma_n$ passing through $e$, $c_{n}$ and $c_{n}^{-1}$, namely
\begin{equation}\label{12-cycle-6}
\begin{array}{lll}
C&=&(e,c_n,(1\ 2)c_n,c_n^{-1}(1\ 2)c_n,c_n^{-2}(1\ 2)c_n,(1\ 2)c_n^{-2}(1\ 2)c_n,c_n(1\ 2)c_n^{-2}(1\ 2)c_n,\\
&&c_n^2(1\ 2)c_n^{-2}(1\ 2)c_n,(1\ 2)c_n^2(1\ 2)c_n^{-2}(1\ 2)c_n,c_n^{-1}(1\ 2)c_n^2(1\ 2)c_n^{-2}(1\ 2)c_n,\\
&&c_n^{-2}(1\ 2)c_n^2(1\ 2)c_n^{-2}(1\ 2)c_n=(1\ 2)c_n^{-1},(1\ 2)c_n^{-2}(1\ 2)c_n^2(1\ 2)c_n^{-2}(1\ 2)c_n=c_n^{-1},\\
&&c_n(1\ 2)c_n^{-2}(1\ 2)c_n^2(1\ 2)c_n^{-2}(1\ 2)c_n=e).
\end{array}
\end{equation}
It is easy to verify that  $C$ is exactly a $12$-cycle.

For IV-type and V-type,  we have $u_1=(3\ 4)$ and $u_4=(n\ 1)$ or $u_1=(2\ 3)$ and $u_4=(n-1\ n)$. In the former case, we get $\{u_2,u_3\}=\{(3\ 4), (n\ 1)\}$. If $u_2=(3\ 4)$, then  $i_2=0$, a contradiction. If $u_2=(n\ 1)$, then $u_3=(3\ 4)$, implying that $i_3=3$, which is impossible. In the later case, one can deduce a contradiction in the same way.

For VI-type,  we have $u_1=(3\ 4)$ and $u_4=(n-1\ n)$. Then $\{u_2,u_3\}=\{(3\ 4), (n-1\ n)\}$. If $u_2=(3\ 4)$, then  $i_2=0$, a contradiction. If $u_2=(n-1\ n)$, then $i_2=-4$, which is impossible.

For VII-type,  we have $u_1=(4\ 5)$ and $u_4=(n\ 1)$ or $u_1=(2\ 3)$ and $u_4=(n-2\ n-1)$. In the former case, we obtain $\{u_2,u_3\}=\{(4\ 5), (n\ 1)\}$. If $u_2=(4\ 5)$, then  $i_2=0$, a contradiction. If $u_2=(n\ 1)$, then $i_2=-4$, which is impossible. In the later case, similarly, one can also deduce a contradiction.

Summarizing the above discussions, we conclude that there is an unique $12$-cycle, which is shown in (\ref{12-cycle-6}), in $\Gamma_n$ passing through $e$, $c_{n}$ and $c_{n}^{-1}$.

We complete the proof.
\end{proof}

\begin{remark}\label{rem-1}
\emph{If $n=12$, there is another $12$-cycle in $\Gamma_n$ passing through $e$, $c_{n}$ and $c_{n}^{-1}$ due to $c_n^{12}=e$. Thus the condition $n\ge 13$ in Lemma \ref{lem-4} is necessary.}
\end{remark}

\begin{lem}\label{lem-5}
Let $S=\{c_n=(1\ 2\ldots n),c_n^{-1},(1\ 2)\}$ and $\Gamma_n=\mathrm{Cay}(S_n,S)$ ($n\geq 13$). Then there are exactly two $12$-cycles in $\Gamma_n$ passing through $e$, $(1\ 2)$ and $c_n$ (resp. $e$, $(1\ 2)$ and $c_n^{-1}$), which are shown in (\ref{12-cycle-12}) and (\ref{12-cycle-14}) (resp. (\ref{12-cycle-15}) and (\ref{12-cycle-16})).
\end{lem}
\begin{proof}
Assume that $C=(e,(1\ 2),s_1(1\ 2),s_2s_1(1\ 2),\ldots,s_{10}\cdots s_2s_1(1\ 2)=c_n,c_{n}^{-1}s_{10}\cdots $ $s_2s_1(1\ 2)=e)$ is an arbitrary $12$-cycle in $\Gamma_n$ passing through $e$, $(1\ 2)$ and $c_n$, where $s_i\in S$ for $1\leq i\leq 10$. Since $C$ is  a cycle, we have $s_1\neq (1\ 2)$, $s_i\neq s_{i-1}^{-1}$ for $2\leq i\leq 10$ and $s_{10}\neq c_n$. Thus there exists a positive integer $k$ such that $C$  is determined by the  equation
\begin{equation}\label{12-cycle-7}
e=c_{n}^{-1}s_{10}\cdots s_2s_1(1\ 2)=c_n^{i_k}(1\ 2)c_n^{i_{k-1}}(1\ 2)\cdots c_n^{i_3}(1\ 2)c_n^{i_2}(1\ 2)c_n^{i_1}(1\ 2),
\end{equation}
where $i_k\le -1$, $|i_l|\ge1$ for $l=1,2,\ldots,k-1$ and $|i_1|+|i_2|+\cdots+|i_{k-1}|-i_k+k=12$.
Clearly, we have $k\le 6$. If $k=1$, then (\ref{12-cycle-7}) is equivalent to $e=c_n^{-11}(1\ 2)$, which is impossible. Therefore, we have $2\leq k\leq 6$. Let  $u_0=(1\ 2)$ and $u_l=c_n^{-i_1-i_{2}-\cdots-i_{l-1}-i_l}(1\ 2)c_n^{i_1+i_{2}+\cdots+i_{l-1}+i_l}$ for $l=1,2,\ldots,k-1$. Then  $u_l=c_{n}^{-i_l}u_{l-1}c_{n}^{i_l}$ for $l=1,\ldots,k-1$, and so (\ref{12-cycle-7}) becomes
\begin{equation}\label{12-cycle-8}
e=c_n^{i_1+i_{2}+\cdots+i_{k-1}+i_k}u_{k-1}\cdots u_3u_2u_1(1\ 2),
\end{equation}
or equivalently,
\begin{equation}\label{12-cycle-9}
c_n^{-i_1-i_{2}-\cdots-i_{k-1}-i_k}=u_{k-1}\cdots u_3u_2u_1(1\ 2).
\end{equation}

Note that $i_1+i_{2}+\cdots+i_{k-1}+i_k\in[-\sum_{l=1}^k|i_l|,\sum_{l=1}^k|i_l|]\subseteq[-12,12]\subset[-n,n]$ because $\sum_{l=1}^k|i_l|\leq 12$ and $n\geq 13$. Therefore, if $i_1+i_{2}+\cdots+i_{k-1}+i_k\not=0$, then $i_1+i_{2}+\cdots+i_{k-1}+i_k\not\equiv0$ $(\mathrm{mod}~n)$, and so $|\mathrm{supp}(c_n^{-i_1-i_{2}-\cdots-i_{k-1}-i_k})|=|\mathrm{supp}(c_n^{i_1+i_{2}+\cdots+i_{k-1}+i_k})|=n\geq 13$. However, $|\mathrm{supp}(u_{k-1}\cdots u_3u_2u_1(1\ 2))|\leq 2(k-1)+2\leq 12$, contrary to (\ref{12-cycle-9}). Thus $i_1+i_{2}+\cdots+i_{k-1}+i_k=0$. From (\ref{12-cycle-8}) we have
\begin{equation}\label{12-cycle-10}
u_{k-1}\cdots u_3u_2u_1(1\ 2)=e.
\end{equation}
Clearly, $k=2$, $4$ or $6$ since $2\leq   k\leq 6$ and $e$ is an even permutation.

If $k=2$, then $|i_1|-i_2=10$ and $i_1+i_2=0$, implying that $i_1=5$ and $i_2=-5$.  Then $u_1=c_{n}^{-5}(1\ 2)c_{n}^{5}=(6\ 7)$, and so $u_1(1\ 2)=(6\ 7)(1\ 2)\neq e$,  contrary to (\ref{12-cycle-10}).

If $k=4$, then $|i_1|+|i_2|+|i_3|-i_4=8$ and $i_1+i_2+i_3+i_4=0$. It is easy to see that there are four types of solutions satisfying these conditions:
\begin{equation*}
\left\{\begin{array}{rl}
\mbox{I-type:}&i_4=-1, ~\{i_1,i_2,i_3\}=\{4,-2,-1\},~\{3,-3,1\}~\mbox{or}~\{-3,2,2\};\\
\mbox{II-type:}&i_4=-2, ~\{i_1,i_2,i_3\}=\{4,-1,-1\},\{3,-2,1\}~\mbox{or}~\{2,2,-2\};\\
\mbox{III-type:}&i_4=-3, ~\{i_1,i_2,i_3\}=\{3,1,-1\}~\mbox{or}~\{2,2,-1\};\\
\mbox{IV-type:}&i_4=-4, ~\{i_1,i_2,i_3\}=\{2,1,1\}.\\
\end{array}\right.
\end{equation*}

For I-type, since $i_4=-1$, we have $u_3=c_n^{-i_1-i_2-i_3}(1\ 2)c_n^{i_1+i_2+i_3}=c_n^{i_4}(1\ 2)c_n^{-i_4}=c_n^{-1}(1\ 2)c_n=(2\ 3)$. Since $u_3u_2u_1(1\ 2)=e$, we get $u_2u_1=(2\ 3)(1\ 2)=(1\ 2\ 3)$, which implies that $u_2=(1\ 2)$ and $u_1=(1\ 3)$, $u_2=(1\ 3)$ and $u_1=(2\ 3)$, or $u_2=(2\ 3)$ and $u_1=(1\ 2)$. The first two cases cannot occur because both $\mathrm{supp}(u_1)$ and $\mathrm{supp}(u_2)$ must contain two consecutive points.  The last case also cannot occur due to $i_1\neq 0$.

For II-type, as above, we have $u_3=(3\ 4)$ and $\{u_1,u_2\}=\{(1\ 2), (3\ 4)\}$. If $u_1=(1\ 2)$, then from $u_1=c_n^{-i_1}(1\ 2)c_n^{i_1}$  we get $i_1=0$, a contradiction. If $u_1=(3\ 4)$, then $u_2=(1\ 2)$.  This gives that $i_1=2$, $i_2=-2$ and $i_3=2$, and so (\ref{12-cycle-7}) becomes
\begin{equation}\label{12-cycle-11}
c_n^{-2}(1\ 2)c_n^{2}(1\ 2)c_n^{-2}(1\ 2)c_n^{2}(1\ 2)=e,
\end{equation}
which holds naturally because $c_n^{-2}(1\ 2)c_n^{2}(1\ 2)=(3\ 4)(1\ 2)$ is of order $2$. Thus there is a possible $12$-cycle in $\Gamma_n$ passing through $e$, $(1\ 2)$ and $c_{n}$, namely
\begin{equation}\label{12-cycle-12}
\begin{array}{lll}
C_1&=&(e,(1\ 2),c_n(1\ 2),c_n^{2}(1\ 2),(1\ 2)c_n^{2}(1\ 2)c_n,c_n^{-1}(1\ 2)c_n^{2}(1\ 2),c_n^{-2}(1\ 2)c_n^{2}(1\ 2),\\
&&(1\ 2)c_n^{-2}(1\ 2)c_n^{2}(1\ 2),c_n(1\ 2)c_n^{-2}(1\ 2)c_n^{2}(1\ 2),c_n^2(1\ 2)c_n^{-2}(1\ 2)c_n^{2}(1\ 2),\\
&&(1\ 2)c_n^2(1\ 2)c_n^{-2}(1\ 2)c_n^{2}(1\ 2)=c_n^2, c_n^{-1}(1\ 2)c_n^2(1\ 2)c_n^{-2}(1\ 2)c_n^{2}(1\ 2)=c_n,\\
&&c_n^{-2}(1\ 2)c_n^2(1\ 2)c_n^{-2}(1\ 2)c_n^{2}(1\ 2)=e).
\end{array}
\end{equation}
It is easy to verify that  $C_1$ is exactly a $12$-cycle.

For III-type, we have $u_3=(4\ 5)$ and $\{u_1,u_2\}=\{(1\ 2), (4\ 5)\}$. If $u_1=(1\ 2)$, then $i_1=0$, a contradiction. If $u_1=(4\ 5)$,  then $u_2=(1\ 2)$, and so $i_1=3$, $i_2=-3$ and $i_3=3$, which is impossible because $\{i_1,i_2,i_3\}=\{3,1,-1\}$ or $\{2,2,-1\}$.

For IV-type, we have $u_3=(5\ 6)$ and $\{u_1,u_2\}=\{(1\ 2), (5\ 6)\}$. If $u_1=(1\ 2)$, then $i_1=0$, a contradiction. If $u_1=(5\ 6)$,   then $i_1=4$, which is impossible.

If  $k=6$, then $|i_1|+|i_2|+|i_3|+|i_4|+|i_5|-i_6=6$ and $i_1+i_2+i_3+i_4+i_5+i_6=0$. Since $i_6\leq -1$ and $|i_l|\geq 1$ for $l=1,\ldots,5$, we have  $i_6=-1$ and $\{i_1,i_2,i_3,i_4,i_5\}=\{1,1,1,-1,-1\}$. Then, by simple computation, we see that the only solution of (\ref{12-cycle-7}) is $(i_1,i_2,i_3,i_4,i_5,i_6)=(1,-1,1,-1,1,-1)$, i.e.,
\begin{equation}\label{12-cycle-13}
c_n^{-1}(1\ 2)c_n(1\ 2)c_n^{-1}(1\ 2)c_n(1\ 2)c_n^{-1}(1\ 2)c_n(1\ 2)=e,
\end{equation}
which holds naturally because $c_n^{-1}(1\ 2)c_n(1\ 2)=(1\ 2\ 3)$ is of order $3$. Therefore, there exists another possible $12$-cycle in $\Gamma_n$ passing through $e$, $(1\ 2)$ and $c_{n}$, that is,
\begin{equation}\label{12-cycle-14}
\begin{array}{lll}
C_2&=&(e,(1\ 2),c_n(1\ 2),(1\ 2)c_n(1\ 2),c_n^{-1}(1\ 2)c_n(1\ 2),(1\ 2)c_n^{-1}(1\ 2)c_n(1\ 2),\\
&&c_n(1\ 2)c_n^{-1}(1\ 2)c_n(1\ 2),(1\ 2)c_n(1\ 2)c_n^{-1}(1\ 2)c_n(1\ 2),\\
&&c_n^{-1}(1\ 2)c_n(1\ 2)c_n^{-1}(1\ 2)c_n(1\ 2),(1\ 2)c_n^{-1}(1\ 2)c_n(1\ 2)c_n^{-1}(1\ 2)c_n(1\ 2),\\
&&c_n(1\ 2)c_n^{-1}(1\ 2)c_n(1\ 2)c_n^{-1}(1\ 2)c_n(1\ 2)=(1\ 2)c_n,\\
&&(1\ 2)c_n(1\ 2)c_n^{-1}(1\ 2)c_n(1\ 2)c_n^{-1}(1\ 2)c_n(1\ 2)=c_n,\\
&&c_n^{-1}(1\ 2)c_n(1\ 2)c_n^{-1}(1\ 2)c_n(1\ 2)c_n^{-1}(1\ 2)c_n(1\ 2)=e).
\end{array}
\end{equation}
It is easy to verify that  $C_2$ is exactly a $12$-cycle.

Summarizing the above discussions, we see that there are exactly two $12$-cycles in $\Gamma_n$ passing through $e$, $c_{n}$ and $c_{n}^{-1}$, namely the cycles $C_1$ and $C_2$ shown in (\ref{12-cycle-12}) and (\ref{12-cycle-14}), respectively.

Similarly, one can show that there are exactly two $12$-cycles in $\Gamma_n$ passing through $e$, $(1\ 2)$ and $c_n^{-1}$, namely the cycles
\begin{equation}\label{12-cycle-15}
\begin{array}{lll}
C_1^*&=&(e,(1\ 2),c_n^{-1}(1\ 2), c_n^{-2}(1\ 2),(1\ 2)c_n^{-2}(1\ 2)c_n^{-1},c_n(1\ 2)c_n^{-2}(1\ 2),\\
&&c_n^2(1\ 2)c_n^{-2}(1\ 2),
(1\ 2)c_n^{2}(1\ 2)c_n^{-2}(1\ 2),c_n^{-1}(1\ 2)c_n^{2}(1\ 2)c_n^{-2}(1\ 2),\\
&&c_n^{-2}(1\ 2)c_n^{2}(1\ 2)c_n^{-2}(1\ 2),
(1\ 2)c_n^{-2}(1\ 2)c_n^{2}(1\ 2)c_n^{-2}(1\ 2)=c_n^{-2}, \\
&&c_n(1\ 2)c_n^{-2}(1\ 2)c_n^{2}(1\ 2)c_n^{-2}(1\ 2)=c_n^{-1},c_n^{2}(1\ 2)c_n^{-2}(1\ 2)c_n^{2}(1\ 2)c_n^{-2}(1\ 2)=e)
\end{array}
\end{equation}
and
\begin{equation}\label{12-cycle-16}
\begin{array}{lll}
C_2^*&=&(e,(1\ 2),c_n^{-1}(1\ 2),(1\ 2)c_n^{-1}(1\ 2),c_n(1\ 2)c_n^{-1}(1\ 2),(1\ 2)c_n(1\ 2)c_n^{-1}(1\ 2),\\
&&c_n^{-1}(1\ 2)c_n(1\ 2)c_n^{-1}(1\ 2),(1\ 2)c_n^{-1}(1\ 2)c_n(1\ 2)c_n^{-1}(1\ 2),\\
&&c_n(1\ 2)c_n^{-1}(1\ 2)c_n(1\ 2)c_n^{-1}(1\ 2),(1\ 2)c_n(1\ 2)c_n^{-1}(1\ 2)c_n(1\ 2)c_n^{-1}(1\ 2),\\
&&c_n^{-1}(1\ 2)c_n(1\ 2)c_n^{-1}(1\ 2)c_n(1\ 2)c_n^{-1}(1\ 2)=(1\ 2)c_n^{-1},\\
&&(1\ 2)c_n^{-1}(1\ 2)c_n(1\ 2)c_n^{-1}(1\ 2)c_n(1\ 2)c_n^{-1}(1\ 2)=c_n^{-1},\\
&&c_n(1\ 2)c_n^{-1}(1\ 2)c_n(1\ 2)c_n^{-1}(1\ 2)c_n(1\ 2)c_n^{-1}(1\ 2)=e)
\end{array}
\end{equation}
which are determined by the equalities
\begin{equation}\label{12-cycle-17}
c_n^{2}(1\ 2)c_n^{-2}(1\ 2)c_n^{2}(1\ 2)c_n^{-2}(1\ 2)=e
\end{equation}
and
\begin{equation}\label{12-cycle-18}
c_n(1\ 2)c_n^{-1}(1\ 2)c_n(1\ 2)c_n^{-1}(1\ 2)c_n(1\ 2)c_n^{-1}(1\ 2)=e
\end{equation}
respectively.

We complete the proof.
\end{proof}

\begin{figure}[h]
\centering
\unitlength 1.7mm 
\linethickness{0.4pt}
\ifx\plotpoint\undefined\newsavebox{\plotpoint}\fi 
\begin{picture}(29,24)(0,0)
\put(15,18){\line(-3,2){6}}
\put(21,22){\line(-3,-2){6}}
\put(15,18){\line(0,-1){6}}
\put(15,12){\line(-3,-2){6}}
\put(9,8){\line(-1,0){6}}
\put(9,8){\line(0,-1){6}}
\put(15,12){\line(3,-2){6}}
\put(21,8){\line(0,-1){6}}
\put(21,8){\line(1,0){6}}
\put(15,18){\circle*{1.5}}
\put(15,12){\circle*{1.5}}
\put(21,8){\circle*{1.5}}
\put(9,8){\circle*{1.5}}
\put(3,8){\circle*{1.5}}
\put(9,2){\circle*{1.5}}
\put(21,2){\circle*{1.5}}
\put(27,8){\circle*{1.5}}
\put(9,22){\circle*{1.5}}
\put(21,22){\circle*{1.5}}
\put(17,12){\makebox(0,0)[cc]{\footnotesize$e$}}
\put(18,18){\makebox(0,0)[cc]{\footnotesize$(1\ 2)$}}
\put(9,24){\makebox(0,0)[cc]{\footnotesize$c_n(1\ 2)$}}
\put(21,24){\makebox(0,0)[cc]{\footnotesize$c_n^{-1}(1\ 2)$}}
\put(11,8){\makebox(0,0)[cc]{\footnotesize$c_n$}}
\put(18,8){\makebox(0,0)[cc]{\footnotesize$c_n^{-1}$}}
\put(29,8){\makebox(0,0)[cc]{\footnotesize$c_n^{-2}$}}
\put(1,8){\makebox(0,0)[cc]{\footnotesize$c_n^2$}}
\put(9,0){\makebox(0,0)[cc]{\footnotesize$(1\ 2)c_n$}}
\put(21,0){\makebox(0,0)[cc]{\footnotesize$(1\ 2)c_n^{-1}$}}
\end{picture}
\caption{\label{fig-1}\footnotesize{Local structure of  $\Gamma_n$.}}
\end{figure}
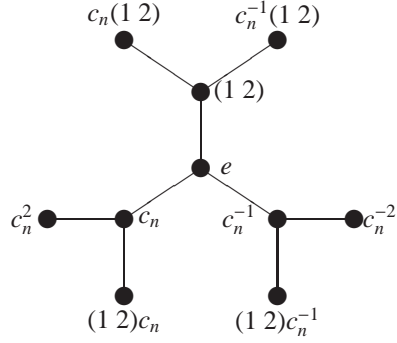

Combining  Lemmas \ref{lem-4} and \ref{lem-5}, we now prove that $\Gamma_n$ is a normal Cayley graph.
\begin{lem}\label{lem-6}
Let $S=\{c_n=(1\ 2\ldots n),c_n^{-1},(1\ 2)\}$ and $\Gamma_n=\mathrm{Cay}(S_n,S)$ ($n\geq 13$). Then $\mathrm{Aut}(\Gamma_n)_e=\mathrm{Aut}(S_n,S)$, or equivalently, $\Gamma_n$ is a normal Cayley graph.
\end{lem}
\begin{proof}
Let $\sigma\in\mathrm{Aut}(\Gamma_n)_e$. Firstly, we claim that $(1\ 2)^\sigma=(1\ 2)$. In fact, if $(1\ 2)^\sigma\neq(1\ 2)$, without loss of generality, we assume that $(1\ 2)^\sigma=c_n$. Then $c_n^{\sigma}=(1\ 2)$ and $(c_n^{-1})^{\sigma}=c_n^{-1}$ or $c_n^{\sigma}=c_n^{-1}$ and $(c_n^{-1})^{\sigma}=(1\ 2)$ because $\sigma\in\mathrm{Aut}(\Gamma_n)_e$ fixes $S$ setwise. Since $\sigma$ sents $12$-cycles to $12$-cycles, by Lemmas \ref{lem-4} and \ref{lem-5}, both of the two cases cannot occur because in $\Gamma_n$ there are only one $12$-cycle  passing through $e$, $c_n$ and $c_n^{-1}$ while there are two $12$-cycles passing through $e$, $(1\ 2)$ and $c_n$ (resp. $e$, $(1\ 2)$ and $c_n^{-1}$).

To prove our result, by Lemmas \ref{lem-1} and \ref{lem-2}, it suffices to show that $(st)^\sigma=s^\sigma t^\sigma$ for any $s,t\in S$. As $(1\ 2)^\sigma=(1\ 2)$, we just need to consider the following two cases.

\textbf{Case 1.} $c_n^\sigma=c_n$ and $(c_n^{-1})^\sigma=c_n^{-1}$;

By Lemma \ref{lem-4}, there is an unique $12$-cycle in $\Gamma_n$ passing through $e$, $c_n$ and $c_n^{-1}$, namely the cycle
$C=(e,c_n,(1\ 2)c_n,\ldots,(1\ 2)c_n^{-1},c_n^{-1},e)$ shown in (\ref{12-cycle-6}). Then $C^\sigma=(e^\sigma,c_n^\sigma,((1\ 2)c_n)^\sigma,$ $\ldots,((1\ 2)c_n^{-1})^\sigma,(c_n^{-1})^\sigma,e^\sigma)=(e,c_n,((1\ 2)c_n)^\sigma,\ldots,((1\ 2)c_n^{-1})^\sigma,c_n^{-1},e)$ is also a $12$-cycle passing through $e$, $c_n$ and $c_n^{-1}$. By the uniqueness of the $12$-cycle, we obtain
\begin{equation}\label{normal-1}
((1\ 2)c_n)^\sigma=(1\ 2)c_n=(1\ 2)^\sigma c_n^\sigma~\mbox{and}~((1\ 2)c_n^{-1})^\sigma=(1\ 2)c_n^{-1}=(1\ 2)^\sigma (c_n^{-1})^\sigma.
\end{equation}
Furthermore,  $\sigma$ fixes $c_n$ and $c_n^{-1}$, and so fixes their neighborhoods $N_{\Gamma_n}(c_n)=\{e, (1\ 2)c_n, c_n^2\}$ and $N_{\Gamma_n}(c_n^{-1})=\{e, (1\ 2)c_n^{-1}, c_n^{-2}\}$ setwise (see Fig. \ref{fig-1}), respectively. Then, by (\ref{normal-1}), we get
\begin{equation}\label{normal-2}
(c_n^2)^\sigma=c_n^2=c_n^\sigma c_n^\sigma~\mbox{and}~(c_n^{-2})^\sigma=c_n^{-2}=(c_n^{-1})^\sigma (c_n^{-1})^\sigma.
\end{equation}
Moreover, by Lemma \ref{lem-5}, there are exactly two $12$-cycles passing through $e$, $(1\ 2)$ and $c_n$, namely  $C_1=(e, (1\ 2), c_n(1\ 2), \ldots, c_n^2, c_n, e)$ and $C_2=(e, (1\ 2), c_n(1\ 2), \ldots, (1\ 2)c_n, c_n, e)$ shown in (\ref{12-cycle-12}) and (\ref{12-cycle-14}), respectively. Note that both $C_1$ and $C_2$ pass through $c_n(1\ 2)$. As $\sigma$ fixes $e$, $(1\ 2)$ and $c_n$, and sents $12$-cycles to $12$-cycles, we have
\begin{equation}\label{normal-3}
(c_n(1\ 2))^\sigma=c_n(1\ 2)=c_n^\sigma(1\ 2)^\sigma.
\end{equation}
Similarly, by considering the $12$-cycles $C_1^{*}$ and $C_2^{*}$ (see (\ref{12-cycle-15}) and (\ref{12-cycle-16})) passing through $e$, $(1\ 2)$ and $c_n^{-1}$, we get
\begin{equation}\label{normal-4}
(c_n^{-1}(1\ 2))^\sigma=c_n^{-1}(1\ 2)=(c_n^{-1})^\sigma(1\ 2)^\sigma.
\end{equation}
Also, it is obvious that
\begin{equation}\label{normal-5}
((1\ 2)^2)^\sigma=e^\sigma=e=(1\ 2)(1\ 2)=(1\ 2)^\sigma(1\ 2)^\sigma.
\end{equation}
Combining (\ref{normal-1})--(\ref{normal-5}), we obtain the  result as required.

\textbf{Case 2.} $c_n^\sigma=c_n^{-1}$ and $(c_n^{-1})^\sigma=c_n$.

Since $\sigma$ swaps $c_n$ and $c_n^{-1}$, as in Case 1, by considering the unique $12$-cycle in $\Gamma_n$ passing through $e$, $c_n$ and $c_n^{-1}$ we get
\begin{align}
&((1\ 2)c_n)^\sigma=(1\ 2)c_n^{-1}=(1\ 2)^\sigma c_n^\sigma,~((1\ 2)c_n^{-1})^\sigma=(1\ 2)c_n=(1\ 2)^\sigma (c_n^{-1})^\sigma; \label{normal-6}\\
&(c_n^2)^\sigma=c_n^{-2}=c_n^\sigma c_n^\sigma,~(c_n^{-2})^\sigma=c_n^{2}=(c_n^{-1})^\sigma (c_n^{-1})^\sigma.\label{normal-7}
\end{align}
Also, by considering the two $12$-cycles $C_1$ and $C_2$ (resp. $C_1^*$ and $C_2^*$) passing through $e$, $(1\ 2)$ and $c_n$ (resp. $e$, $(1\ 2)$ and $c_n^{-1}$), we obtain
\begin{equation}\label{normal-8}
(c_n(1\ 2))^\sigma=c_n^{-1}(1\ 2)=c_n^\sigma(1\ 2)^\sigma~\mbox{and}~(c_n^{-1}(1\ 2))^\sigma=c_n(1\ 2)=(c_n^{-1})^\sigma(1\ 2)^\sigma.
\end{equation}
Combining (\ref{normal-5})--(\ref{normal-8}), we obtain the  result as required.

The proof is now complete.
\end{proof}

\begin{remark}\label{rem-2}
\emph{From the proof of Lemma \ref{lem-6} we see that each $\sigma\in\mathrm{Aut}(\Gamma_n)_e$ $(n\geq 13)$ must fix $(1\ 2)\in S$. Thus $\mathrm{Aut}(\Gamma_n)_e$ is not transitive on the neighborhood of the identity vertex $e$, which implies that $\Gamma_n$ is not arc-transitive.}
\end{remark}

By Lemma \ref{lem-6}, $\Gamma_n$ ($n\geq 13$) is a normal Cayley graph, so the full automorphism group of $\Gamma_n$ is equal to $\mathrm{Aut}(\Gamma_n)=R(S_n)\rtimes \mathrm{Aut}(S_n,S)$ by Godsil \cite{Godsil}. Thus, in order to determine $\mathrm{Aut}(\Gamma_n)$, it suffices to determine the group $\mathrm{Aut}(S_n,S)$.  The following lemma completely determine the group $\mathrm{Aut}(S_n,S)$.

\begin{lem}\label{lem-7}
Let $S_n$ be the symmetric group of degree $n$, and let $S=\{c_n=(1\ 2\ldots\ n),c_n^{-1},$ $(1\ 2)\}$ ($n\geq 3$). Then $$\mathrm{Aut}(S_n,S)=\langle\mathrm{Inn}(\phi)\rangle,$$
where $\phi=(1\ 2)(3\ n)(4\ n-1)(5\ n-2)\cdots$ $(\in S_n)$, and $\mathrm{Inn}(\phi)$ denotes the inner isomorphism of $S_n$ induced by $\phi$.
\end{lem}
\begin{proof}
Let $\varphi\in \mathrm{Aut}(S_n,S)$. Then $\varphi\in \mathrm{Aut}(S_n)$ and $S^\varphi=S$. Since  $\varphi$ cannot sent $(1\ 2)\in S$ to a product of three disjoint transpositions (we only need to consider this situation when $n=6$),  from Lemma \ref{lem-3} we know that $\varphi\in \mathrm{Inn}(S_n)$, and so there exists some $\phi\in S_n$ such that $\varphi=\mathrm{Inn}(\phi)$. Thus $S^\varphi=\phi^{-1}S\phi=S$, that is,
\begin{equation}\label{Aut-1}
\begin{array}{rl}
&\{\phi^{-1}(1\ 2\ 3\ \ldots\ n)\phi,\phi^{-1}(1\ n\ n-1\ \ldots\ 2)\phi,\phi^{-1}(1\ 2)\phi\}\\
=&\{(\phi(1)\ \phi(2)\ \phi(3)\ \ldots\ \phi(n)),(\phi(1)\ \phi(n)\ \phi(n-1)\ \ldots\ \phi(2)),(\phi(1)\ \phi(2))\}\\
=&\{(1\ 2\ 3\ \ldots\ n),(1\ n\ n-1\ \ldots\ 2),(1\ 2)\}.
\end{array}
\end{equation}
According to (\ref{Aut-1}), we have $(\phi(1)\ \phi(2))=(1\ 2)$. Therefore,  $\phi(1)=1$ and $\phi(2)=2$ or $\phi(1)=2$ and $\phi(2)=1$. Again by (\ref{Aut-1}), the former case implies that  $\phi=e$ while the later case implies that $\phi=(1\ 2)(3\ n)(4\ n-1)(5\ n-2)\cdots$. It follows our result.
\end{proof}

By Lemmas \ref{lem-6} and \ref{lem-7}, we obtain the main result of this paper immediately.

\begin{thm}\label{thm-main}
Let $S=\{c_n=(1\ 2\ldots n),c_n^{-1},(1\ 2)\}$ and $\Gamma_n=\mathrm{Cay}(S_n,S)$ ($n\geq 13$). Then $$\mathrm{Aut}(\Gamma_n)=R(S_n)\rtimes \mathrm{Aut}(S_n,S)=R(S_n)\rtimes \langle\mathrm{Inn}(\phi)\rangle\cong S_n\rtimes \mathbb{Z}_2,$$
where $R(S_n)$ is the right regular representation of $S_n$, $\phi=(1\ 2)(3\ n)(4\ n-1)(5\ n-2)\cdots$ $(\in S_n)$, and $\mathrm{Inn}(\phi)$ is the inner isomorphism of $S_n$ induced by $\phi$.
\end{thm}

\begin{remark}\label{rem-3}
\emph{It is worth mentioning  that the conclusion of Theorem \ref{thm-main} also holds for $\mathrm{Cay}(S_n,S)$ with $S=\{c_n,c_n^{-1},(i\ i+1)\}$.}
\end{remark}

\begin{remark}\label{rem-4}
\emph{Noting that Theorem \ref{thm-main} gives the automorphism group of $\Gamma_n$ for $n\geq 13$. For $n\leq 12$, by using  the package ``grape'' of GAP4 \cite{GAP4}, we obtain that $\mathrm{Aut}(\Gamma_n)\cong D_{6}$ if $n=3$ and $\mathrm{Aut}(\Gamma_n)\cong S_n\times \mathbb{Z}_2$ if $4\leq n\leq 8$; however, to determine the automorphism group of $\Gamma_n$ for $9\leq n\leq 12$ is beyond the capacity of our computer.}
\end{remark}

\end{document}